\documentclass[12pt]{elsarticle} %[reqno]{amsart}
\usepackage{graphicx}
\usepackage{amsmath}
\usepackage{bm}
\usepackage{amsfonts}
\usepackage{amsthm}
\usepackage{placeins}
\usepackage{afterpage}
\newtheorem{Theorem}{Theorem}[section]

\def\ee{\end{eqnarray*}}
\def\be{\begin{eqnarray*}}
\def\bee{\end{eqnarray}}
\def\bbe{\begin{eqnarray}}
\def\ea{\end{align*}}
\def\ba{\begin{align*}}

\makeatletter
\let\today\relax
\def\ps@pprintTitle{%
    \let\@oddhead\@empty
    \let\@evenhead\@empty
    \def\@oddfoot{\footnotesize\itshape
      \hfill\today}%   {Preprint submitted} \hfill\today}%
    \let\@evenfoot\@oddfoot
    }
\makeatother

\makeatletter
\newcommand{\bal}{\@ifstar{\@bals}{\@bal}}
\def\@bals#1\eal{\begin{align*}#1\end{align*}}
\def\@bal#1\eal{\begin{align}#1\end{align}}
\makeatletter

\def\u{\bm u}

\def\Z{\mathbb Z}

    \def\T{ \mathbb{T}^2}
    
    \def\p{\partial}

\begin{document}
\journal{}
\title{ Quasi-stationary solutions  of the  surface quasi-geostrophic equation  }
% Force line breaks with \\

\author{Zhi-Min Chen}

 \ead{zmchen@szu.edu.cn}
\address{School  of Mathematics and Statistics, Shenzhen University, Shenzhen 518060,  China}%

%\tableofcontents\emph{}
%\date{}% It is always \today, today,

\begin{abstract}
In the present study, we find that  the surface quasi-geostrophic equation admits exact solutions, which evolve with time in quasi-stationary states.
The solutions presented are available    for any dissipation effect  $\kappa (-\Delta)^\alpha$ ($\kappa >0$,  $0\le \alpha <1$), involved in  the equation.
When   the equation is supercritical ($0\le \alpha <\frac12$), the problem on the existence of large global  regular solutions remains open.
This study, however, provides  explicit sample solutions for the understanding of the uncertain problem.

\end{abstract}
\begin{keyword}
Surface quasi-geostrophic equation, dissipative  quasi-geostrophic equation, exact solutions, quasi-stationary states, special solutions
\end{keyword}

\maketitle

\section{Introduction}
Consider the surface quasi-geostrophic equation
\bbe
0=\p_t\theta +\u\cdot \nabla \theta +\kappa (-\Delta )^\alpha \theta
\label{new1}
\bee
for $0\leq \alpha <1$ and  the dissipative parameter $\kappa >0$. Here $\theta$ is a scalar unknown representing potential temperature and  $\u$ is the velocity expressed as
\be \u  = (\p_y, -\p_x) (-\Delta)^{-\frac12}\theta.
\ee

Equation (\ref{new1}) presents  an interesting simple model that exhibits a number of nonlinear and
dissipative characters of the 3D Navier-Stokes equations, and hence has been studied extensively. For the subcritical case $\frac12<\alpha <1$, the existence of global  solutions was obtained by Constantin and Wu \cite{Con}.  When   $\alpha =\frac12$, equation (\ref{new1}) is critical, as it is comparable to the 3D Navier-Stokes equations with respect to a priori estimates in function spaces. However, the maximal principle, which  is not applicable to  the 3D Navier-Stokes equations, is available to (\ref{new1}). Thus global  regular solutions remain existing (see Kiselev {\it et al.} \cite{K}).   For the supercritical case $0\le \alpha <\frac12$, the existence of local regular solutions and small global  regular solutions have been obtained by Chae and Lee \cite{Lee}, Chen {\it et al.} \cite{Chene}, Wu \cite{Wu} and C\'ordoda and C\'orododa \cite{Co} in a variety of function spaces. However, it is unknown for the existence of global regular solutions in the supercritical case. If the motion (\ref{new1}) is additionally   driven by an external force, the existence of bifurcating stationary flows  was studied by Chen and Price \cite{Chen08} and the author \cite{Chen16}.

For the understanding of the dynamical behaviors of  the solution to (\ref{new1}), we consider  the equation to be   $2\pi$-spatially periodic  in the domain $\T=[0, 2\pi)\times [0,2\pi)$, and simply present some exact solutions, which evolve in  quasi-stationary  states.

 \section{Exact solutions}

 The exact solution result is stated in the following.
 \begin{Theorem}\label{th}
 Let $0\le \alpha <1$ and $\kappa>0$.
For any  real constants $c_1,..., c_8$, and integers $n,\,m$ with $nm\ne0 $ and $k$ so that
\bbe n^2+m^2=k^2 \,\,\mbox{ whenever }\,\, (|c_1|+...+|c_4|)(|c_1|+...+|c_4|)\ne 0,
\bee
  then (\ref{new1}) in $\T$  admits the following exact solution
\bal
\theta =& e^{-\kappa (n^2+m^2)^\alpha t} \Big(c_1\sin  n x \sin my+c_2\cos  n x \sin my+c_3\sin  n x \cos my\nonumber
\\ \nonumber
& \hspace{22mm}+c_4\cos  n x \cos my\Big)\\
&+ e^{-\kappa |k|^{2\alpha} t}\Big( c_5\sin kx+c_6\sin ky+c_7 \cos kx+c_8 \cos ky\Big). \label{s1}
\eal
Moreover,  for any integers $n$ and $m$  with $|n|+|m|\ne0 $ and for any constants $a_k$ and $b_k$ so that $\sum _{k \in \Z} |k|(a_k^2 + b_k^2)<\infty$, then the function
\bal \theta &= \sum_{k\in \Z}  e^{-\kappa (n^2k^2+m^2k^2)^\alpha t}\Big(a_k\cos ( kn x+kmy) + b_k\sin ( kn x+kmy)\Big) \label{s2}
\eal
solves (\ref{new1}) in $\T$.

\end{Theorem}

\begin{proof}

The proof of (\ref{th}) is straightforward.

For the validity   with respect to $\theta$ given by (\ref{s1}), we see that
\be (-\Delta)^{-\frac12 } \theta = \lambda \theta  \,\,\mbox{ for either }\,\, \lambda = \frac1{\sqrt{n^2+m^2}} \,\,\mbox{ or } \,\, \lambda = \frac1{|k|}.
\ee
This implies
\be \u \cdot \nabla \theta % &=& (\p_y (-\Delta)^{-\frac12}\theta, -\p_x (-\Delta)^{-\frac12}\theta)\cdot \nabla \theta \\
&=&\lambda (\p_y \theta, -\p_x \theta)\cdot \nabla \theta=0.
\ee
Rewrite (\ref{s1}) in the following form
$\theta(t) = e^{-(-\Delta)^{\alpha}t} \theta (0),$
which solves the linear equation
\bbe \label{L}\p_t\theta + \kappa (-\Delta)^\alpha \theta =0
\bee
and hence solves (\ref{new1}).

For the validity with  respect to $\theta$ given by (\ref{s2}), we see that
\bal* \u\cdot \nabla \theta =&\Big(\p_y (-\Delta)^{-\frac12}\theta, -\p_x (-\Delta)^{-\frac12}\theta\Big)\cdot \nabla \theta
\\
=&\sum_{k,k'\in \Z}e^{-\kappa (|k'|^{2\alpha} +|k|^{2\alpha})(n^2+m^2)^\alpha t}\frac{km (-a_k\varphi_k+b_k \phi_k) n  k'(-a_{k'} \varphi_{k'} +b_{k'}\phi_{k'})}{\sqrt{n^2k^2+m^2k^2}}
\\
-&\sum_{k,k'\in \Z} e^{-\kappa (|k'|^{2\alpha} +|k|^{2\alpha})(n^2+m^2)^\alpha t}\frac{kn (-a_k\varphi_k+b_k \phi_k) m k'(-a_{k'} \varphi_{k'} +b_{k'}\phi_{k'})}{\sqrt{n^2k^2+m^2k^2}}
\\
=&0
\eal
for
\be \phi_k=\cos ( kn x+kmy) \,\,\mbox{ and }\,\, \varphi_k=\sin ( kn x+kmy).
\ee
Note that $\theta$ solves the linear equation (\ref{L}) and hence solves (\ref{new1}).

The proof is complete.
\end{proof}

When $\kappa$ is  as small as $0.001$ and $\alpha$ is close to the critical stage $\frac12$, satisfactory numerical solutions of (\ref{new1}) via  a spectral scheme were presented  by Constantin {\it et al.} \cite{Con2}
by choosing respectively  the following initial data
\bal*
&\sin x\sin y +\cos y,
\\
&-\cos 2x \cos y + \sin x \sin y,
\\
& \cos 2x\cos y+\sin x \sin y + \cos 2x\sin 3y.
\eal
As shown in \cite{Con2}, the flow patterns of the  solutions initially from these  data vary with the  time $t$.

In the present study, however,  the flow patterns of the exact solutions behave in a quasi-stationary manner.
For example,  we choose the solutions
\bal
\theta_1&= e^{-5^\alpha \kappa t} ( \sin 2x\sin y +\frac12\cos 2x\cos y),\label{s3}
\\
\theta_2&=e^{-25^\alpha \kappa t} ( \sin 4x\sin 3y +\frac12\cos 4x\cos 3y +\sin 5y +\frac12 \sin 5x)\label{s4}
\eal
from (\ref{s1}), and the solution
\bal
\theta_3 &= e^{-2^\alpha \kappa t}  \sin (x+y) +  e^{-8^\alpha \kappa t}  \sin (2x+2y)\label{s5}
\eal
from (\ref{s2}). They are displayed in Figure \ref{f1} for $\kappa=\alpha =0.001$.

\begin{figure}[h]
 \centering
 \includegraphics[height=.35\textwidth, width=.49\textwidth]{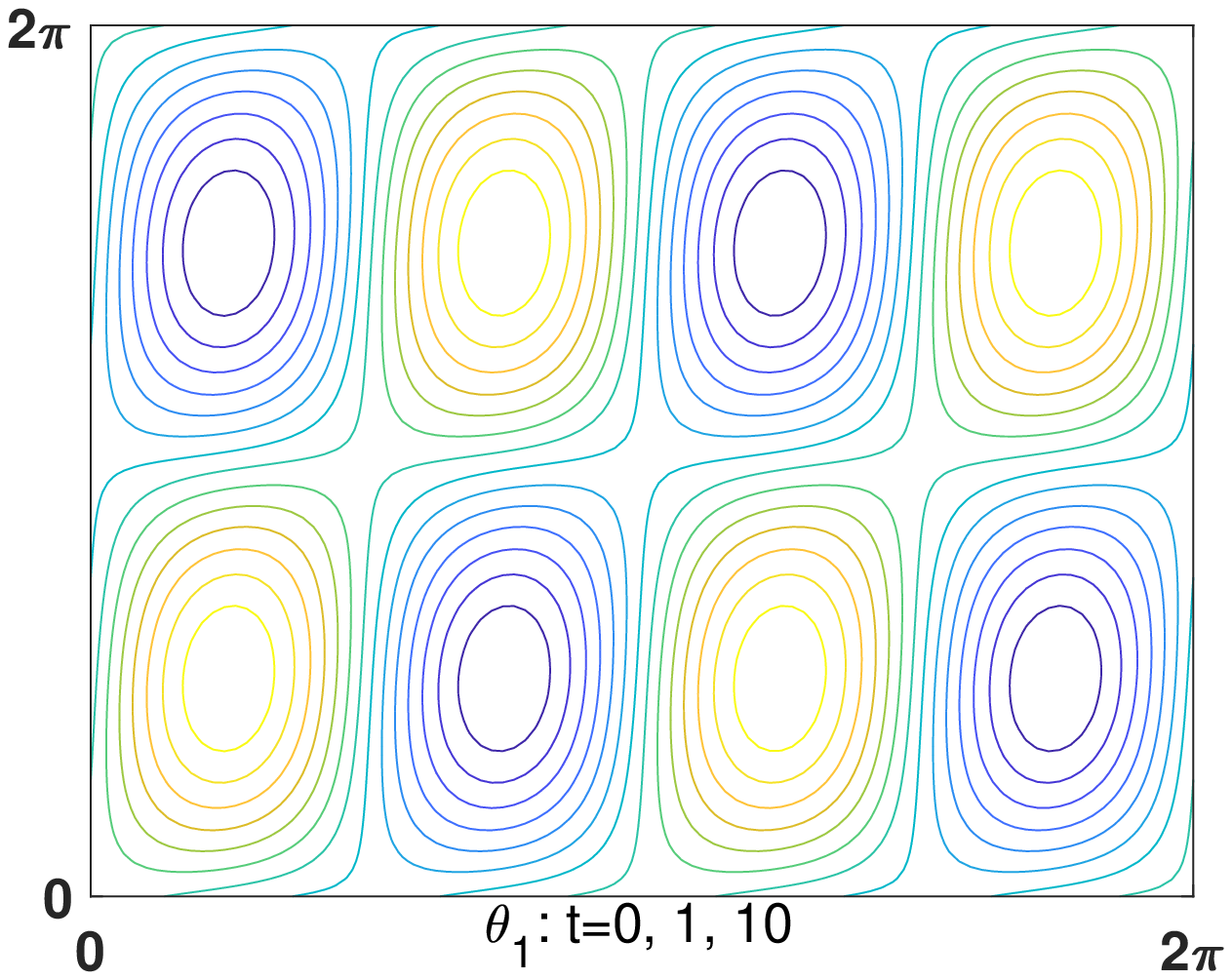}
 \includegraphics[height=.35\textwidth, width=.49\textwidth]{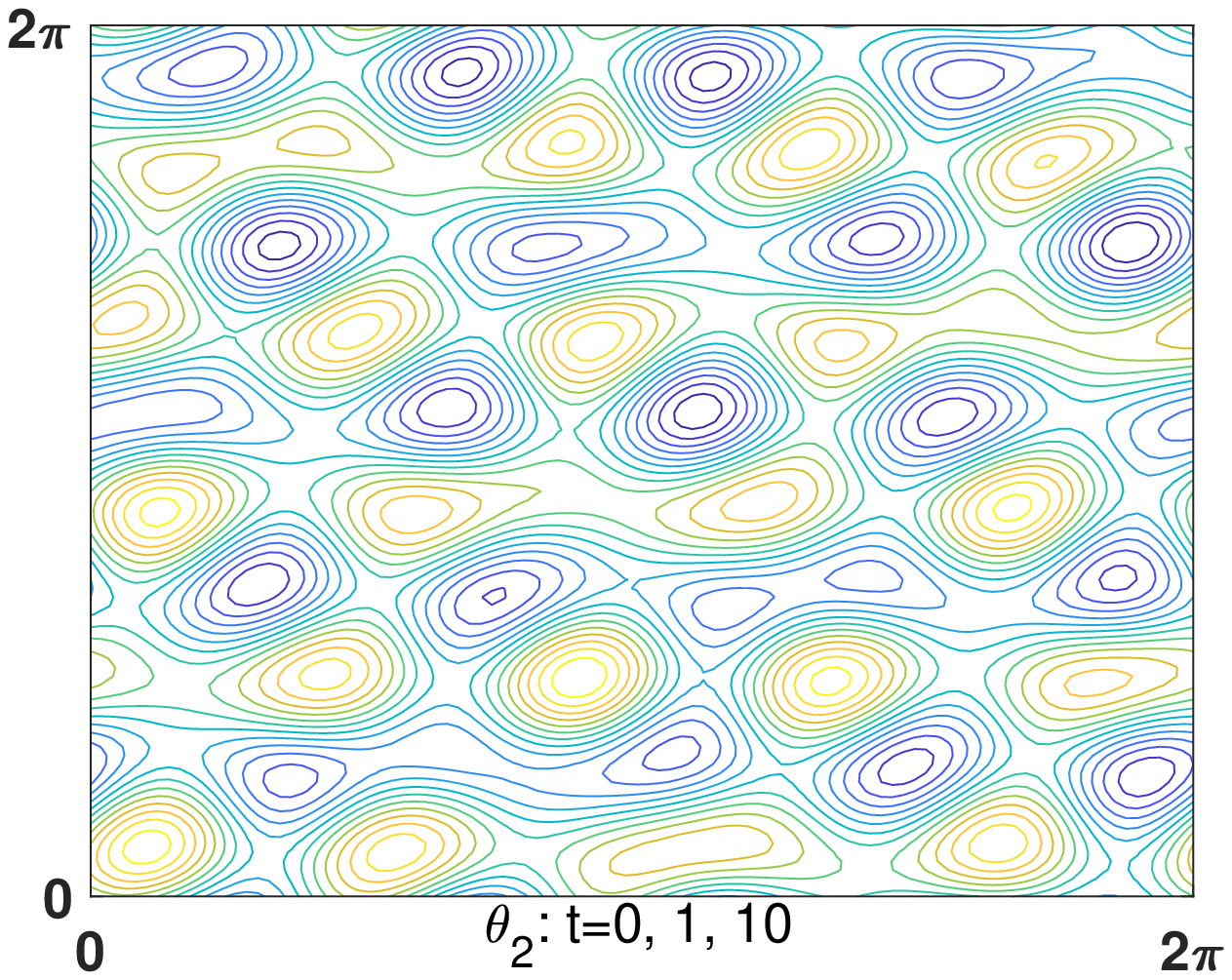}
 \includegraphics[height=.35\textwidth, width=.49\textwidth]{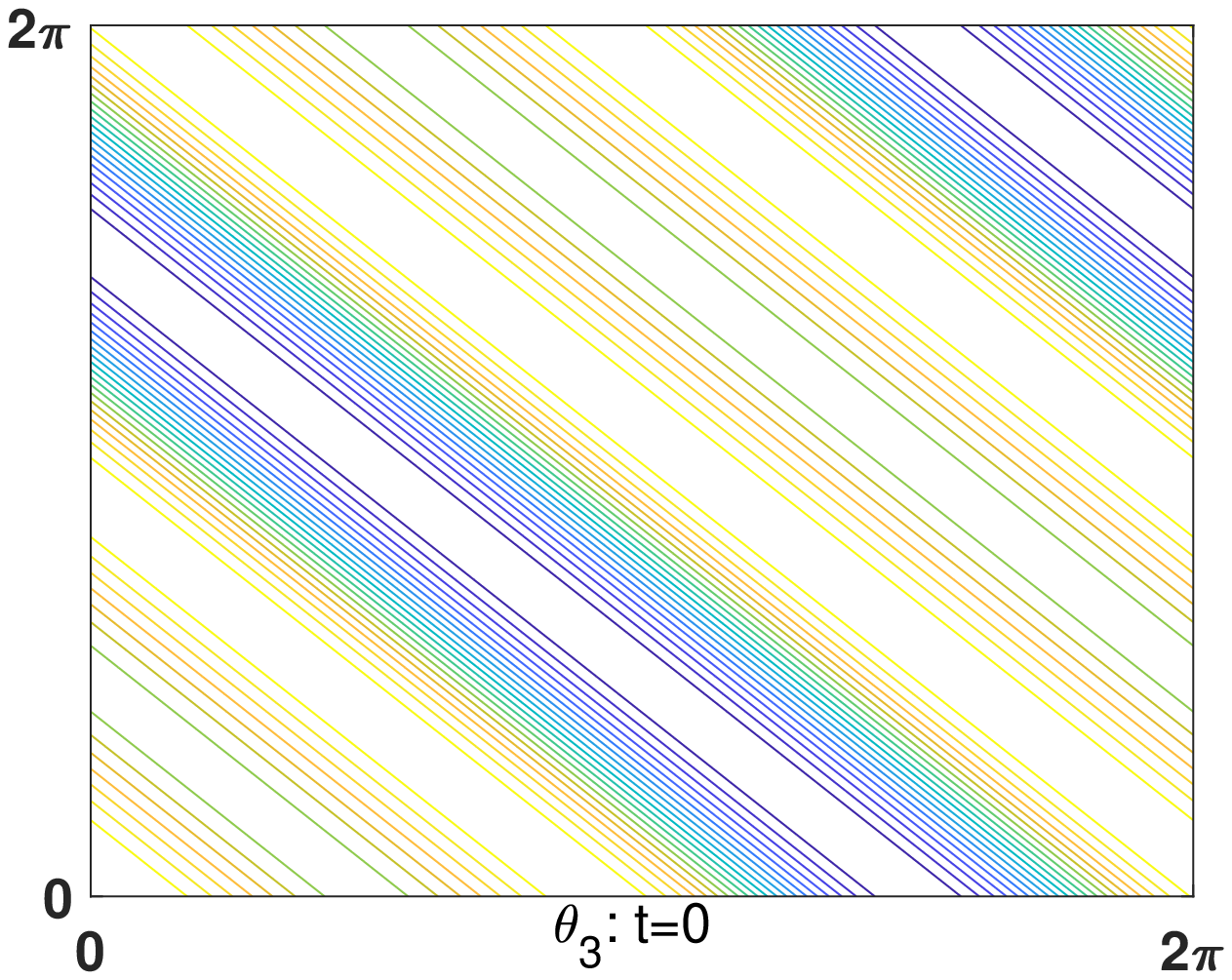}
  \includegraphics[height=.35\textwidth, width=.49\textwidth]{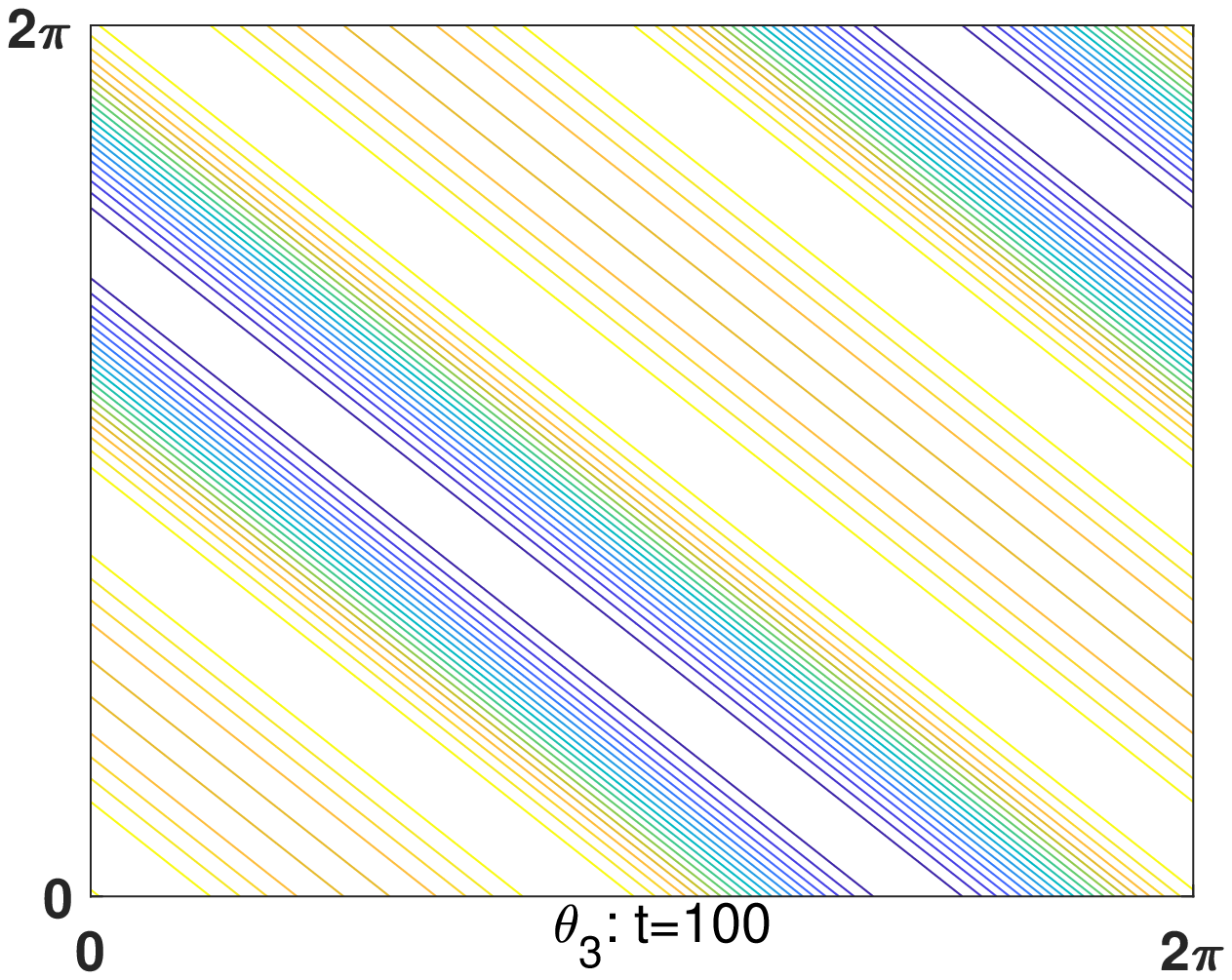}
\caption{Exact solutions $\theta_1$, $\theta_2$ and $\theta_3$ for $\alpha =0.001$ and $\kappa =0.001$.
}
 \label{f1}
 \end{figure}

 The solutions $\theta_1$ and $\theta_2$ as well as the solution (\ref{s1}) are quasi-stationary in the sense that their  flow patterns remain unchanged (see Figure \ref{f1}) as $t$ grows.  This is due to the fact that the flow patterns defined by $\theta(t)=$constants are the same with those defined by $\theta(0)=$constants. Therefore, the flow patterns are not sensitive with the change of the parameters $\alpha$ and $\kappa$. The solution $\theta$ in (\ref{s2}) is a unidirectional flow moving along the parallel straight lines
 \be nx+my= {\rm constants}.
 \ee
 Figure \ref{f1} also shows that the solution $\theta_3$  as well as the solution (\ref{s2}) evolves in a quasi-stationary manner, as the flow patterns of $\theta_3$ for $t>0$  remain parallel to their initial form.

It should be noted that the exact solutions are for any $\kappa>0$ and $\alpha\ge 0$. In the numerical computation  such as \cite{Con2}, it is difficult to keep the spectral scheme convergence when $\alpha$ is close to $0$.
The  vortex flow  $e^{-\kappa (n^2+m^2)^\alpha t}\sin nx \sin m y$ is developed  from  the Taylor flow \cite{T}.  The present study is developed from the author's recent investigation \cite{Chen21} on  metastability of Kolmogorov flow. 

\

\noindent \textbf{Acknowledgment}. This research was partially supported by NSFC of China (Grant No. 11571240).

\

\end{document}